\documentclass[a4paper,12pt]{amsart}
\usepackage{graphicx,amsmath,amsfonts,amssymb,amsthm,mathrsfs}
\usepackage{geometry}
\usepackage{float}
\usepackage{hyperref}
\usepackage{setspace}
\usepackage{quiver}
\usepackage{xcolor}
\usepackage{listings}
\usepackage[utf8]{inputenc}
\lstdefinelanguage{Magma}{
  keywords={for, while, if, do, then, else, function, return},
  sensitive=true,
  morecomment=[l]{//},
  morecomment=[s]{/*}{*/},
  morestring=[b]",
}
\hypersetup{
    colorlinks=true,
    linkcolor=blue,
    filecolor=magenta,      
    urlcolor=cyan,
    pdftitle={Overleaf Example},
    pdfpagemode=FullScreen,
    }

\newcommand{\FF}{{\mathbb F}}

\def\Fq2{{\mathbb F}_{q^2}}
\def\Fp2{{\mathbb F}_{p^2}}

\def\cC{{\mathcal C}}
\def\cF{{\mathcal F}}
\def\cH{{\mathcal H}}
\def\cX{{\mathcal X}}
\def\cG{{\mathcal G}}

\def\cY{{\mathcal Y}}

\def\Fq2{{\mathbb F}_{q^2}}

\def\cY{{\mathcal Y}}

\def\cC{{\mathcal C}}
\def\cF{{\mathcal F}}
\def\cH{{\mathcal H}}
\def\cX{{\mathcal X}}

\def\cG{{\mathcal G}}
\DeclareMathOperator{\lcm}{lcm}
\DeclareMathOperator{\ord}{ord}
\numberwithin{equation}{section}
\theoremstyle{plain}
\newtheorem{thm}{Theorem}
\numberwithin{thm}{subsection}

\newtheorem{cor}[thm]{Corollary}
\newtheorem{pro}[thm]{Proposition}
\newtheorem{definition}[thm]{Definition}

\newtheorem{remark}[thm]{Remark}

\newtheorem{example}[thm]{Example}

\def\blfootnote{\xdef\@thefnmark{}\@footnotetext}
\title{Locally recoverable codes with multiple recovering sets from maximal curves}
\author{Saeed Tafazolian  and Jaap Top}
\address{Departamento de Matem\'{a}tica - Instituto de Matem\'{a}tica, Estat\'{i}stica e Computação Cient\'{i}fica
(IMECC) - Universidade Estadual de Campinas (UNICAMP), Rua S\'{e}rgio Buarque de Holanda, 651, Cidade Universit\'{a}ria,  Zeferino Vaz, Campinas, SP 13083-859, Brazil}

\address{Bernoulli Institute for Mathematics, Computer Science, and Artificial Intelligence\\
Nijenborgh~9\\9747 AG Groningen\\ the Netherlands}
\email{saeed@unicamp.br}
\email{j.top@rug.nl}
\begin{document}

\begin{abstract}
In this paper, we present a construction of locally recoverable codes (LRCs) with multiple recovery sets using algebraic curves with many rational points. 
By leveraging separable morphisms between smooth projective curves and expanding the class of curves previously considered, 
we significantly generalize and enhance the framework. 
Our approach corrects certain inaccuracies in the existing literature while extending results to a broader range of curves, 
thereby achieving better parameters and wider applicability. In addition, the constructions
presented here result in LRCs with large availability.
\end{abstract}

\maketitle

\section{Introduction}

The concept of locally recoverable codes (LRCs) arises from the need to enhance the reliability and efficiency of distributed storage systems. An error-correcting code \( C \subseteq \mathbb{F}_q^n \) is said to have \emph{locality} \( r \) if the value of any coordinate of a codeword can be determined by accessing at most \( r \) other coordinates. Such codes enable the local correction of a single erasure without requiring access to the entire codeword.

In distributed storage systems, LRCs with small locality \( r \) are particularly useful because they minimize the number of data nodes involved in the repair process. However, when multiple erasures occur—for example, within the same local repair set—recovery may fail. To address this, codes with \emph{availability} \( t \) have been introduced. These codes associate each coordinate with \( t\) disjoint recovery sets, ensuring robustness against multiple erasures. 

Algebraic geometry has provided powerful tools for constructing LRCs, notably, as initiated in, e.g., \cite{BTV, BHHMV, HMM, JMX, LMX, LM,  SVV}, by using algebraic curves and surfaces defined over finite fields. In \cite{JKZ} the authors  extended these ideas to codes with multiple recovery sets using rational function fields. This approach was generalized by Bartoli and others in \cite{BML} to curves of higher genus. However, several arguments and computations contain flaws, and therefore, the related claims regarding parameters of the codes they construct in some cases require adjustments.

Nevertheless, the work in \cite{BML} presents a powerful general method for constructing LRCs from algebraic curves with many rational points, utilizing group actions to produce codes with multiple recovering sets. Notably, Theorems~IV.1 and~V.1 of \cite{BML} offer a broad framework for generating such codes, potentially with more than one recovering set. However, their explicit constructions—focused mainly on certain maximal curves such as Hermitian and Norm-Trace—typically lead to codes with availability two or three.

Our approach builds upon and significantly extends this framework. We generalize the construction to a wider class of curves
(for a brief comparison, see Remark~\ref{?????}),  and we correct or refine aspects of the results in \cite{BML} (see, e.g., Remarks~3.2.3, 3.2.5, and~3.3.4). Furthermore, we present new families of LRCs with improved parameters. Notably, our method allows for the construction of codes with much higher availability—for example, availability $q + 1$ in Example~3.3.3—depending on the automorphism group and the structure of the underlying curve. These results highlight the enhanced flexibility and strength of our general framework (see also Remarks~2.2.1 and~3.2.2).

For storage system applications, it is desirable to have codes with low locality, high rate, and large minimum distance. Our approach successfully achieves these properties. Additionally, the LRCs we construct offer improved parameters, including better minimum distance bounds and higher availability compared to previous constructions. These contributions underscore the crucial role of maximal curves in designing high-performance LRCs for distributed storage systems.

The paper is organized as follows. In Section 2, we provide an overview of the necessary background on locally recoverable codes, algebraic geometry codes, maximal curves, and their automorphism groups. In Section 3, we make the
general setup described in \cite[Sections~IV and V]{BML}
more explicit by applying it to a concrete and wide class of curves. This is illustrated by presenting several constructions of locally recoverable codes with good parameter estimates.

\section{Preliminaries}

\subsection{Locally Recoverable Codes (LRCs) with Availability}
A block code is said to have \textit{locality} \(r\) if each coordinate of a codeword can be recovered by accessing at most \(r\) other coordinates. Formally, the definition of a  LRC with locality \(r\) is given below.

\begin{definition}
Let \(C \subseteq \mathbb{F}_q^n\) be a (\(q\)-ary) code. For each \(\alpha \in \mathbb{F}_q\) and \(i \in \{1, 2, \dots, n\}\), define:
\[
C(i, \alpha) := \{ \mathbf{c} = (c_1, c_2, \dots, c_n) \in C \mid c_i = \alpha \}.
\]
Given a nonempty subset \(I \subseteq \{1, 2, \dots, n\} \setminus \{i\}\), denote by \(C_I(i, \alpha)\) the projection of \(C(i, \alpha)\) in \(\FF_q^I\). The code \(C\) is called a \textit{locally recoverable code} with locality \(r\) if, for every \(i \in \{1, 2, \dots, n\}\), there exists \(I_i \subseteq \{1, 2, \dots, n\} \setminus \{i\}\) with \(1\leq |I_i| \leq r\) such that for any \(\alpha, \beta \in \mathbb{F}_q\) with \(\alpha \neq \beta\), the sets \(C_{I_i}(i, \alpha)\) and \(C_{I_i}(i, \beta)\) are disjoint.
\end{definition}

In this paper, we focus on LRCs that are linear over \(\mathbb{F}_q\). A \(q\)-ary locally recoverable code of length \(n\), dimension \(k\), minimum distance \(d\), and locality \(r\) is referred to as an \([n, k, d]_q\)-locally recoverable code with locality \(r\), or simply as an \([n, k, d; r]\)-code over \(\mathbb{F}_q\).

An important extension of LRCs is the concept of \textit{availability}, which allows a coordinate to have multiple disjoint recovery sets. A code is said to have \textit{availability} \(t > 1\) if every coordinate \(i\) of a codeword can be recovered from \(t\) disjoint recovery sets. Writing \(r_{i,1}, r_{i,2}, \dots, r_{i,t}\) for the cardinalities of these respective recovery sets, such a code is denoted by
\[
[n,k,d\, ;\,r_1,\cdots,r_t].
\]

 When $t = 1$, it was established that the minimum distance of such a code satisfies the upper bound  
\begin{equation} \label{eq:single_recovery}
    d \leq n - k - \left\lceil \frac{k}{r} \right\rceil + 2.
\end{equation}

For the general case where \( t \geq 1 \) and \( r = r_1 = \dots = r_t \), in \cite{TB} the bound in \eqref{eq:single_recovery} was extended to  
\begin{equation} \label{eq:multiple_recovery}
    d \leq n - \sum_{i=1}^{t} \left\lfloor \frac{k - 1}{r_i} \right\rfloor.
\end{equation}

In \cite[Theorem~III.2]{BML}, the authors generalize \eqref{eq:single_recovery} to  
\begin{equation} \label{eq:generalized_bound}
    d \leq n - k - \left\lceil \frac{(k - 1)t + 1}{1 +\sum_{i=1}^{t}  r_i} \right\rceil + 2.
\end{equation}

 In analogy with the definition of the singleton defect for linear codes, the authors in \cite{BML} introduce the concept of relative defect for a \([n, k, d; r_1, \dots, r_t]\)-code. This is defined using the above bound in (\ref{eq:generalized_bound})  as follows:

\[
\Delta(C) = \frac{1}{n} \left( n - k - d + 2 - \left\lceil \frac{(k - 1)t + 1}{1 +\sum_{i=1}^{t}  r_i} \right\rceil \right).
\]
Conceptually, a smaller relative defect indicates a better code.

In this work, we focus on constructing locally recoverable codes with multiple availability \(t>1\) using algebraic curves defined over \( \mathbb{F}_q \), particularly maximal curves and groups of automorphisms of such curves.

\subsection{Algebraic Geometry Codes and LRCs}\label{subsec2.6}

Algebraic geometry codes, introduced by Goppa, are constructed from the rational points of algebraic curves over finite fields. By leveraging the structure of these curves, these codes achieve long lengths and large minimum distances, often approaching the Gilbert-Varshamov bound, making them highly efficient for data transmission and storage (see \cite{vLvdG} for an introduction).

Moreover, algebraic geometry codes provide a foundation for constructing locally recoverable codes (LRCs) with excellent performance. In this subsection, we present a construction of LRCs based on a separable morphism between smooth projective curves, following a variant of the methods in \cite{BTV, BHHMV, CKMTW, BML, HMM, JKZ, JMX, BK, LMX, MP, MTT}.

Algebraic Geometry (AG) codes are a class of linear codes constructed from algebraic curves over finite fields. Let $\mathcal{X}$ be a smooth, projective, and absolutely irreducible algebraic curve defined over a finite field $\mathbb{F}_q$. The function field of $\mathcal{X}$ is denoted by $\mathbb{F}_q(\mathcal{X})$.

Consider a divisor 
\[
D = P_1 + P_2 + \cdots + P_n,
\]
where the $P_i$ are distinct $\mathbb{F}_q$-rational points on $\mathcal{X}$, and another divisor $G$ whose support is disjoint from the points $P_i$.

The AG code $C_L(D,G)$ is defined as the image of the evaluation map
\[
\begin{aligned}
\mathrm{ev} : \mathcal{L}(G) &\to \mathbb{F}_q^n, \\
f &\mapsto (f(P_1), f(P_2), \ldots, f(P_n)),
\end{aligned}
\]
where 
\[
\mathcal{L}(G) = \{ f \in \mathbb{F}_q(\mathcal{X})^\times \cup \{0\} : (f) + G \geq 0 \} \cup \{0\}
\]
is the Riemann-Roch space associated to the divisor $G$.

The parameters of the code $C_L(D,G)$ depend on the degree of the divisors and the genus $g$ of the curve $\mathcal{X}$. In particular, its dimension satisfies
\[
k = \dim \mathcal{L}(G) \geq \deg(G) - g + 1,
\]
and its minimum Hamming distance is bounded below by
\[
d \geq n - \deg(G).
\]

\begin{remark}
    \rm{Throughout this paper, by \emph{curve} we mean a smooth (non-singular), projective, and absolutely irreducible algebraic curve defined over a finite field.  It is a well-known classical fact that any (absolutely irreducible)
    variety $\mathcal{V}$ of dimension $1$ over a (perfect) field $k$ is birational over $k$ to a unique (upto isomorphisms over $k$)
    such smooth and projective curve $\cC$. As is customary, in various examples we present and describe such $\mathcal{V}$
    whereas we in fact always work with the curve $\cC$.}
\end{remark}

Let $\cX$ and $\cY$ be smooth, projective, and absolutely irreducible curves defined over $\FF_q$. 
Moreover, let \( \varphi\colon \cX \to \cY \) be 
a separable morphism of degree \( r+1 \) defined over $\FF_q$.

Let \( Q_1, Q_2, \dots, Q_s \) be points in \( \cY(\FF_q) \) that split completely in the covering \( \cX \to \cY \). This means that for each \( Q_i \), there exist \( r+1 \) points \( P_{i,0}, P_{i,1}, \dots, P_{i,r} \in \cX(\FF_q) \) such that \( \varphi(P_{i,j}) = Q_i \) for all \( j \).  
Choose a function \( z \in  \FF_q(\cX) \) such that \( e_1 = 1, e_2 = z, \dots, e_r = z^{r-1} \) are elements that are linearly independent over \( \FF_q(\cY) \), and whose poles do not include any of the points \( P_{i,j} \). Moreover, assume that for any fixed $i$ the values $z(P_{i,j})$ are pairwise distinct. Additionally, let \( f_0,f_1, f_2, \dots, f_{t-1} \in \FF_q(\cY) \) be elements that are linearly independent over \( \FF_q \), and whose poles do not include any of the points \( Q_i \).

Take $m\leq r-1$ and consider the vector space
\[
V := \left\{ \sum_{i=0}^{m} \left( \sum_{j=0}^{t-1} a_{i,j} \varphi^*(f_j) \right) z^i \mid a_{i,j} \in \FF_q \right\}.
\] 
Define the linear map
\[
e\colon V \to \FF_q^n,
\]
by
\[
f \mapsto e(f) = (f(P_1), \dots, f(P_n)),
\]
where \( \{P_1, \dots, P_n\} \) is the set of rational points in \( \cX(\FF_q) \) lying above the points \( Q_j \).

Let \( D \) be the smallest effective divisor on \( \cX \) such that \( V \subseteq \mathcal{L}(D) \), where \( \mathcal{L}(D) \) is the Riemann-Roch space associated with the divisor \( D \). Denote the degree of \( D \) by \( \delta \). 

If \( \delta < n = s(r+1) \), then \( e(V) \subseteq \FF_q^n \) is a linear code with parameters
\[
[n = s(r+1), k = (m+1)t, d \geq s(r+1) - \delta].
\]
This construction yields a locally recoverable code with locality \( m+1 \). Specifically, for each \( i = 1, 2, \dots, s \), the set
\[
H_i = \varphi^{-1}(Q_i) = \{ P_{i,j} \mid j = 0, 1, \dots, r \}
\]
forms a \textit{helper set}, which allows recovery of any point within \( H_i \) using the other points in the set, as follows. 

Assume that \( P \) is a rational point lying above \( Q_i \). For any \( f \in V \), we assert that \( f(P) \) can be determined using the values \( f(P_k) \) for \( P_k \in H_i \setminus \{P\} \).
Write \( f \) in the form
\[
f = \sum_{\ell=0}^{m} \left( \sum_{j=0}^{t-1} a_{\ell j} \varphi^*(f_j) \right) z^\ell,
\]
where the coefficients \( a_{\ell j} \) belong to \( \FF_q \). Now, define the function
\[
g = \sum_{\ell=0}^{m} b_\ell z^\ell,
\]
where \( b_\ell = \sum_{j=0}^{t-1} a_{\ell j} f_j(P) \in \FF_q \). Notice that \( f(P) = g(P) \) and \( f(P_k) = g(P_k) \) for each \( P_k \in H_i \setminus \{P\} \).
Since \( g(z) \) is a polynomial of degree at most \( m \), and the values \( g(P_k) = f(P_k) \) are known for \( k = 1, 2, \dots, r \) and moreover the values $z(P_k)$ are pairwise distinct, the polynomial \( g(z) \) can be reconstructed uniquely by Lagrange interpolation using only $m+1$ points $P_k$. Consequently, \( f(P) \) is determined as \( f(P) = g(P) \).
This shows that the code $e(V)$ has parameters
\[
[n = s(r+1), k = (m+1)t, d \geq s(r+1) - \delta\,;\,m+1].
\]
\begin{remark}\label{availability}{\rm The argument presented here in fact shows: there exist $\left\lfloor \frac{r}{m+1} \right\rfloor$ recovery sets, in other words, we have a LRC with availability $\left\lfloor \frac{r}{m+1} \right\rfloor$. Namely, in the set $H_i\setminus \{P\}$ any subset of cardinality $m+1$ is a helper set.}
\end{remark}

\subsection{Maximal Curves}

For a smooth projective curve $\cX$ of genus \( g \) defined over $\FF_q$, the number of rational points \( \#\cX(\mathbb{F}_q) \) satisfies the Hasse-Weil bound
\[
\#{\cX}(\mathbb{F}_q) \leq q + 1 + 2g\sqrt{q}.
\]

A curve that reaches this upper bound is called a \textit{maximal curve} over $\FF_q$. These curves are of particular interest in coding theory because they provide a large number of rational points, which can be used to construct algebraic geometry codes with better performance in terms of minimum distance and dimension.

In this paper, we focus mostly on the use of Kummer curves for constructing good LRCs, generalizing the methods used for elliptic curves and rational functions to higher-genus curves. This allows us to explore a broader class of codes and improve upon the limitations of existing constructions.

We will also use \cite[Theorem~3.2]{SFN}, which states the following. \begin{pro} 
    \label{max1} Let $q$ be a prime power$,$ $n\geq 2$ an 
   integer, and $f(x)\in \FF_{q^2}[x]$ a separable 
polynomial of degree $m\geq 2$ with 
$\gcd(n,m)=1$ and $\gcd(q,n)=1.$ Let 
   $\cX$ be the non-singular model over $\FF_{q^2}$ of the  curve defined 
   by $y^n=f(x)$. If $\cX$ is maximal over $\Fq2$, 
then $f(x)$ has a root in $\FF_{q^2}$ if and only if $n$ divides $q+1$. 
In this case$,$ all the roots of $f(x)$ belong to $\FF_{q^2}.$
    \end{pro}

\subsection{Galois subcovering a curve}

Let \( \mathcal{X} \) be a curve and \( \cG \) a subgroup of its automorphism group. Consider the Galois quotient \( \mathcal{X}/\cG \). 
For a point \( P \in \mathcal{X} \) and its image \( Q \in \mathcal{X}/\cG \) under the natural projection, 
the orbit $\cG \cdot P:=\{\sigma (P) \mid \sigma \in \cG\}$ of \( P \) under the action of \( \cG \) equals the set of points 
in \( \mathcal{X} \) that map to \( Q \) in \( \mathcal{X}/\cG \).
Therefore, if \( Q \) completely splits in the extension (i.e., the natural projection does not ramify over $Q$), the cardinality of the orbit of \( P \) is equal to the cardinality of \( \cG \). In terms of function fields, assume $\cX$ is defined over
$\FF_q$ and write $\cF=\FF_q(\cX)$. In the situations discussed in the remainder of this paper we moreover assume
that all automorphisms in $\cG$ are defined over the base field $\FF_q$. The extension $\cF\supset \cF^{\cG}=\FF_q(\cX/\cG)$
is Galois with group $\cG$. The rational places of $\cF^{\cG}$ that split completely in $\cF$ correspond to the
orbits $\cG\cdot P$ of size $\#\cG$ consisting of rational places of $\cF$. Equivalently, these are the orbits under
the action of $\cG$ of the rational places $P$ of $\cF$
with the property $\textrm{Stab}_{\cG}(P)=\{\textit{id}\}$.


\section{Explicit Constructions of LRCs  from Certain Curves}

The LRCs described in this section obtain their
helper sets from the existence of several subgroups
of the automorphism group of the curves used
in constructing the LRC. Following the setup in \cite{BML}, the text is organized
by distinguishing how these subgroups intersect.

\subsection{Subgroups with trivial intersection}

 \begin{thm}\label{311} Let $\mathcal{F}$ be the function field of the curve over $\FF_q$ defined by the equation
\[
F(y) = A(x),
\]
where \( A(x) \) and \( F(y) \) are polynomials of degrees \( \alpha \) and \( \beta \), respectively, with \( \gcd(\alpha, \beta) = 1 \). 

Consider two disjoint subgroups \( \mathcal{H}_1 \) and \( \mathcal{H}_2 \) in $\text{\rm Aut}_{\FF_q}(\cF)$ with 
\(\# \mathcal{H}_1 = \alpha\) and $\#\mathcal{H}_2=\beta$ such that \( \mathcal{F}^{\mathcal{H}_1} = \mathbb{F}_q(y) \) and \(  \mathcal{F}^{\mathcal{H}_2}=\mathbb{F}_q(x)  \).
Assume that \( \mathcal{G} := \mathcal{H}_1 \mathcal{H}_2 \) is a group. 

Suppose that \( {Q}_1, \ldots, {Q}_s \) are rational places of \( \mathcal{F}^{\mathcal{G}} \) that split completely in \( \mathcal{F} \),
and moreover each of the functions $x$ and $y$ 
attains pairwise different values on the places over any $Q_j$. 
The number of such places of $\mathcal{F}$ over the ${Q}_j$'s equals  \( s\alpha\beta \).

For any \( t_1 \leq \beta-2 \) and \( t_2 \leq \alpha-2\) such that $\alpha t_1+\beta t_2 < s\alpha\beta$, there exists a 
\[
[s\alpha\beta, (t_1 + 1)(t_2 + 1), d\geq s\alpha\beta - \alpha t_1 - \beta t_2; t_2+1, t_1+1]
\]
LRC  over \( \mathbb{F}_q \),
obtained by evaluating a suitable $\FF_q$-subspace of $\mathcal{F}$ in
the places of $\mathcal{F}$ over $\{Q_1,\ldots,Q_s\}$.

\end{thm}

\begin{proof}
Let \( P_\infty \) be the unique pole of \( x \) and  \( y \in \mathcal{F} \). Then, \( (y)_\infty = \alpha P_\infty \) and \( (x)_\infty = \beta P_\infty \).  For \( t_i \geq 0 \), \( i = 1, 2 \), consider the \( \mathbb{F}_q \)-vector spaces

\[
V_1 := \left\{
    \sum_{k=0}^{\alpha-2} \left( \sum_{j=0}^{t_1} a_{k,j} y^{j} \right) x^k : a_{k,j} \in \mathbb{F}_{q}
\right\}
\]
and
\[
V_2 := \left\{
    \sum_{k=0}^{\beta-2} \left( \sum_{j=0}^{t_2} a_{k,j} x^j \right) y^k : a_{k,j} \in \mathbb{F}_{q}
\right\}.
\]
 Note that the inequalities $t_1 \leq \beta-2$ and $t_2 \leq \alpha-2$ imply \( \dim_{\mathbb{F}_{q}} V_1 = (\alpha-1)(t_1 + 1) \) and
 \( \dim_{\mathbb{F}_{q}} V_2 = (\beta-1)(t_2 + 1) \).

The inequalities also imply that the vector space
\[
V := V_1 \cap V_2 = \left\{
    \sum_{k=0}^{t_1} \left( \sum_{j=0}^{t_2} a_{k,j} x^j \right) y^{k} : a_{k,j} \in \mathbb{F}_{q}
\right\}
\]
and
\[
\dim_{\mathbb{F}_{q}} V = (t_1 + 1)(t_2 + 1).
\]
Now, let \( m := \alpha t_1 + \beta t_2 \). Then we obtain the inclusion \( V \subset L(m P_\infty) \).

Consider, for $n:=s\alpha\beta$, the linear map
\[
e\colon V \to \mathbb{F}_q^n,
\]
given by
\[
f \mapsto e(f) = (f(P_1), \dots, f(P_n)),
\]
where \( \{P_1, \dots, P_n\} \) is the set of places of $\cF$ above the ${Q}_j$'s.

The linear code \( e(V) \) is contained in the algebraic geometry code \( C_L\left( \sum_{i=1}^{n} P_i,  m P_\infty \right) \). Observe that by assumption $n > t_1\alpha+t_2\beta=m$ which shows
that $e$ is injective, so $\dim_{\FF_q}e(V)=(t_1+1)(t_2+1)$.

Then, the minimum distance of the code \( e(V) \) is greater than or equal to the minimum distance of the linear code \( C_L\left( \sum_{i=1}^{n} P_i, m P_\infty \right) \). In the case at hand this yields \(  d \geq n-m \). 

Considering the general construction in Subsection \ref{subsec2.6}, we know that an error in the coordinate at \( P=P_j \) can be repaired by any subset of cardinality \( t_2+1,t_1+1 \) for \( i = 1, 2 \), respectively of the two sets 
\[
R_i = \cH_i\cdot P\setminus \{P\}.
\]
Finally, we have \( {\mathcal{H}_1\cdot P}\, \cap \,{\mathcal{H}_2\cdot P} = \{ P \} \) since by assumption the action of $\cH_1\cH_2$ on the orbit of $P$ is faithful. 
\end{proof}
\begin{remark}{\rm
   As in Remark~\ref{availability}, the proof
   presented here in fact constructs LRCs with
   availability $\lfloor\frac{\alpha-1}{t_2+1}\rfloor + \lfloor\frac{\beta-1}{t_1+1}\rfloor$. An analogous observation holds
   for all subsequent  results in this section.
}\end{remark}
 
\begin{cor}
 Let $\mathcal{X}$ be a curve over $\mathbb{F}_q$ defined by the equation \( y^\beta = A(x) \), where $\beta\geq 2$ is a divisor of $q-1$, and $A(x)$ is an additive and separable polynomial of degree $\alpha$ that splits in $\FF_q[x]$. Assume $n=\#\mathcal{X}(\mathbb{F}_q) - \alpha - 1>0$. For any $t_1 \leq \beta - 2$ and $t_2 \leq \alpha - 2$, there exists a recoverable code $\mathcal{C}$ over $\mathbb{F}_q$ with parameters  
\[
\big[n, \, (t_1 + 1)(t_2 + 1), \, d\geq n - \alpha t_1 - \beta t_2; \, t_2+1, \, t_1+1 \big].
\]

\end{cor}

\begin{proof}
   Consider the function field \( \mathcal{F} = \mathbb{F}_{q}(x, y) \) with
\[
y^{\beta} =A(x).
\]
Since $\beta \mid q-1$ and all roots of $A(x)$ belong to $\FF_q$, we have the following two subgroups of \( \mathbb{F}_{q}\)-automorphisms of \( \mathcal{F} \), where a $\lambda\in\FF_q^\times$ of order $\beta$ is fixed:
\[
\mathcal{H}_1 := \left\{ (x, y) \mapsto (x + c, y) \;|\; c \in \mathbb{F}_{q}, \; A(c) = 0 \right\},
\]
\[
\mathcal{H}_2 := \left\{ (x, y) \mapsto (x, \lambda^i y) \;|\; i \in\mathbb{Z} \right\}.
\] 
Observe that \( \#\mathcal{H}_1 = \alpha \), \( \#\mathcal{H}_2 = \beta \), \( \mathcal{F}^{\mathcal{H}_1} = \mathbb{F}_{q}(y) \), and
\( \mathcal{F}^{\mathcal{H}_2} = \mathbb{F}_{q}(x) \). The groups \( \mathcal{H}_1 \) and \( \mathcal{H}_2 \) have trivial intersection and commute, thus \( \cG = \mathcal{H}_1 \mathcal{H}_2 \cong \mathcal{H}_1 \times \mathcal{H}_2 \) is a group of cardinality $\alpha \beta$ and \( \mathcal{F}^\cG = \mathbb{F}_{q}( y^{\beta})=\mathbb{F}_{q}(A(x)) \subseteq \mathcal{F}^{\mathcal{H}_i} \) for \( i = 1, 2 \).

Let \( P_{\infty} \) be the unique pole of \( x \) and \( y \). Then, we have
\[
(y)_{\infty} = \alpha P_{\infty} \quad \text{and} \quad (x)_{\infty} = \beta P_{\infty}.
\]
Considering the function field extension \( \mathbb{F}_q(x, y) / \mathbb{F}_q(x) \), we have:  
\[
\# \cX(\mathbb{F}_q) = 1 + \alpha + m \alpha \beta,
\]
for some integer \(m>0\).   
Indeed, \(\mathcal{X}\) has one point at infinity, \(\alpha\) rational points corresponding to the zeros of \(A(x)\), and all other rational points yield a $\cG$-orbit of cardinality $\alpha\beta$.  
Therefore, in the extension \( \mathcal{F} / \mathbb{F}_{q}(y^{\beta}) \), the number of places that are totally split is given by  
\[
\#\cX(\FF_q)-1 - \alpha
\]
which is $>0$ by assumption.
\end{proof}

\begin{cor} \label{C4.1}
Let $\mathcal{X}$ be a maximal curve over $\mathbb{F}_{q^2}$ defined by the equation \( y^\beta = A(x) \), where \( A(x) \) is an additive and separable polynomial of degree $\alpha$. For any $t_1 \leq \beta - 2$ and $t_2 \leq \alpha - 2$, there exists a LRC  over $\mathbb{F}_{q^2}$ with parameters  
\[
\big[n, \, (t_1 + 1)(t_2 + 1), \, d\geq n - \alpha t_1 - \beta t_2 \,;\, t_2+1, \,t_1+1 \big]
\]
where $n=q^2 + (\alpha - 1)(\beta - 1)q - \alpha$.
\end{cor}

\begin{proof}
Since $\cX$ is maximal over $\mathbb{F}_{q^2}$, from Proposition \ref{max1} we know that $\beta$ divides $q + 1$ and all roots of \( A(x) \) belong to \(\mathbb{F}_{q^2}\), and 
\[
\#\cX(\FF_{q^2})=1 + q^2 + 2gq
\]
where \( g = (\alpha - 1)(\beta - 1)/2 \) is the genus of \(\cX\).  

 Therefore, in the extension \( \FF_{q^2}(\cX) \) over \( \mathbb{F}_{q^2}(y^\beta) \), the number of totally split rational places equals 
\[
q^2 + (\alpha - 1)(\beta - 1)q - \alpha.
\]

\end{proof}

\begin{remark}\label{?????}
\rm{The corollaries above generalize Propositions IV.4, IV.8, and IV.9 in \cite{BML}, which pertain to the Hermitian function field, the generalized Hermitian curve, and the Norm-Trace curve, respectively. Furthermore, this result can be applied to a broad class of maximal curves, e.g., by starting from examples presented in \cite{GF, AT, ABQ,SFN}. 
Additionally, similar arguments can be given to construct LRCs from other Kummer maximal curves.

}
\end{remark}

\begin{example}
\rm
This is an example illustrating Theorem~\ref{311}. Consider the Kummer curve \(\mathcal{X}\) over \(\mathbb{F}_{q^2}\) defined by the equation
\[
y^{\beta} = x^{\alpha} + 1,
\]
where \(\alpha, \beta \geq 2\) are integers dividing \(q + 1\). Then, by \cite{FT}, \(\mathcal{X}\) is maximal over \(\mathbb{F}_{q^2}\). We further assume that \(\gcd(\alpha, \beta) = 1\).

Fix elements \(\gamma, \lambda \in \mathbb{F}_{q^2}^*\) of orders \(\alpha\) and \(\beta\), respectively. Consider two subgroups of the automorphism group of \(\mathcal{X}\), defined by:
\[
\mathcal{H}_1 := \left\{ \sigma_i : (u, v) \mapsto (\gamma^i u, v) \;\middle|\; i = 0, \dots, \alpha - 1 \right\},
\]
\[
\mathcal{H}_2 := \left\{ \tau_j : (u, v) \mapsto (u, \lambda^j v) \;\middle|\; j = 0, \dots, \beta - 1 \right\},
\]
and let \(\mathcal{G} := \mathcal{H}_1 \mathcal{H}_2\), which is a group.

Since \(\mathcal{X}\) is maximal over \(\mathbb{F}_{q^2}\), we have:
\[
\#\mathcal{X}(\mathbb{F}_{q^2}) = 1 + q^2 + (\alpha - 1)(\beta - 1)q.
\]

Let \(P = (u, v) \in \mathcal{X}(\FF_{q^2})\) be a rational point that is neither at infinity nor satisfies \(u = 0\) or \(v = 0\). The orbit of \(P\) under the action of \(\mathcal{G}\) is long (i.e., has size \(\#\mathcal{G} = \alpha \beta\)). Consequently, there are
\[
n := \#\mathcal{X}(\FF_{q^2}) - 1 - \alpha - \beta = q^2 + (\alpha - 1)(\beta - 1)q - \alpha - \beta
\]
rational points that split completely in the extension \(\mathcal{F} / \mathcal{F}^{\mathcal{G}}\), where \(\mathcal{F}\) is the function field of \(\mathcal{X}\).

Hence, for any integers \(t_1 \leq \beta - 2\) and \(t_2 \leq \alpha - 2\), this construction yields an explicit LRC over \(\mathbb{F}_{q^2}\) with parameters:
\[
\left[ n, \; (t_1 + 1)(t_2 + 1), \; d \geq n - \alpha t_1 - \beta t_2 \, ; \; t_2 + 1, \; t_1 + 1 \right].
\]
\end{example}

\vspace{\baselineskip}
The next result and examples illustrate extensions of
the presented results to situations with
$\cF^{\cH_1}\neq \FF_q(y)$ or $\cF^{\cH_2}\neq \FF_q(x)$. The proof requires only minor adaptations
compared to the earlier ones. 

\begin{pro} \label{P4.}
   Let $\mathcal{X}$ be a maximal curve  over $\FF_{q^2}$ defined by the equation $y^{\beta}=A(x)$, where $A(x)$ is an additive and separable polynomial with degree $\alpha$. Suppose $u\geq 2$ is a divisor of $\beta$. For any $t_1 \leq u-2$ and $t_2 \leq \alpha-2 $, there exists a 

\[
[n, (t_1+1)(t_2+1), d\geq n-\alpha t_1-\beta t_2\,;\; t_2+1, t_1+1]
\]
LRC over $\FF_{q^2}$, where  $n=q^2 + (\alpha - 1)(\beta - 1)q - \alpha$.
\end{pro}

\begin{proof}
   Consider the function field \( \mathcal{F} = \mathbb{F}_{q^2}(x, y) \) with
\[
y^{\beta} =A(x).
\]
Since $\mathcal{X}$ is maximal over $\FF_{q^2}$,  from Proposition  \ref{max1} we know that $\beta$ divides $q+1$ and all roots of 
$A(x)$ are rational. 
Therefore we have the following two subgroups of \( \mathbb{F}_{q^2}\)-automorphisms of \( \mathcal{F} \):

\[
\mathcal{H}_1 := \left\{ (x, y) \mapsto (x + c, y) \;|\; c \in \mathbb{F}_{q^2}, \; A(c) = 0 \right\},
\]
\[
\mathcal{H}_2 := \left\{ (x, y) \mapsto (x, \lambda^i y) \;|\; i = 0, \dots, u - 1 \text{ and } \lambda \in \mathbb{F}_{q^2}^* \text{ of order } u \geq 2 \right\}.
\] 
It is easy to see that \( \#\cH_1 = \alpha \), \( \#\cH_2 = u \), \( \mathcal{F}^{\cH_1} = \mathbb{F}_{q^2}(y) \), and
\( \mathcal{F}^{\cH_2} = \mathbb{F}_{q^2}(x,y^u) \). The groups \( \cH_1 \) and \( \cH_2 \) have trivial intersection and commute, thus \( \cG = \cH_1 \cH_2 \cong \cH_1 \times \cH_2 \) is a group of cardinality $\alpha u$ and \( \mathcal{F}^G = \mathbb{F}_{q^2}( y^{u}) \subseteq \mathcal{F}^{\cH_i} \) for \( i = 1, 2 \). In the extension \( \mathcal{F} | \mathbb{F}_{q^2}(y^u) \), we have  \( q^2 + (\alpha - 1)(\beta - 1)q-\alpha\) places which are totally split.

Let \( P_{\infty} \) be the unique pole of \( x \) and \( y \) in \( \mathcal{F} \). Then, we have
\[
(y)_{\infty} = \alpha P_{\infty} \quad \text{and} \quad (x)_{\infty} = \beta P_{\infty}.
\]

For \( t_i \geq 0 \), \( i = 1, 2 \), consider the \( \mathbb{F}_{q^2} \)-vector spaces

\[
V_1 := \left\{
    \sum_{k=0}^{\alpha-2} \left( \sum_{j=0}^{t_1} a_{k,j} y^{j} \right) x^k : a_{k,j} \in \mathbb{F}_{q^2}
\right\}
\]
and
\[
V_2 := \left\{
    \sum_{k=0}^{u-2} \left( \sum_{j=0}^{t_2} a_{k,j} x^j \right) y^k : a_{k,j} \in \mathbb{F}_{q^2}
\right\}.
\]

\[
V := V_1 \cap V_2 = \left\{
    \sum_{k=0}^{t_1} \left( \sum_{j=0}^{t_2} a_{k,j} x^j \right) y^{k} : a_{k,j} \in \mathbb{F}_{q^2}
\right\}.
\]
 Then
\( \dim_{\mathbb{F}_{q^2}} V = (t_1 + 1)(t_2 + 1) \). 
Writing $m:=\alpha t_1 + \beta t_2$, we get $V \subseteq L(m P_{\infty}) $.
\end{proof}

\begin{example}\label{ex3.1.8}
    \rm{Consider the Hermitian curve \( y^{65} = x^{64} + x \) over \( \mathbb{F}_{4096} \). Let \( u = 5 \), \( t_1 = 3 \), and \( t_2 = 62 \). By the above proposition, there exists a \( 4096 \)-ary LRC with parameters \([262080, 252, d \geq 257793; 63, 4]\) and relative defect \(\Delta = 0.01537\).  

For comparison, see the example in \cite[Table II]{BML}, where they provide an LRC with parameters \([262080, 64, d \geq 253952; 63, 4]\) and relative defect \(\Delta = 0.03076\). This demonstrates that our code, while maintaining the same length, achieves a significantly higher dimension and minimum distance, as well as a smaller relative defect.
}
\end{example}

\begin{example}
\rm{The Suzuki curve is a maximal curve defined over finite fields of characteristic \(2\), specifically over the field \(\mathbb{F}_{q^4}\), where \(q = 2^{2t+1}\) and \(t \geq 0\) is any integer. This has interesting properties, including a large number of \(\mathbb{F}_{q}\)-rational points relative to its genus. The Suzuki curve is described by the affine equation:
\[
y^q - y = x^{q_0}(x^q - x),
\]
where \(q_0  = 2^{t}\).  Its genus is $q_0(q-1)$. 
Matthews has constructed certain codes based on Suzuki curves; see \cite{GM}. 
Write $\mathcal{S}$ for the
function field of the Suzuki curve over $\FF_{q^4}$. Let \( P_\infty \) be the unique pole of \( x \) and  \( y \in \mathcal{S} \). Then,

\[ 
(x)_{\infty} = q P_{\infty} \quad \text{and} \quad (y)_{\infty} = (q+q_0)P_{\infty}.
\]

Define two subgroups of the automorphism  group of $\mathcal{S}$, by 
\[
\cH_1 := \left\{ (x, y) \mapsto ( x, y+\lambda) \;|\;   \lambda \in \mathbb{F}_{q}  \right\},
\]
\[
\cH_2 := \left\{ (x, y) \mapsto (\gamma x, \gamma^{q_0+1} y) \;|\;  \gamma \in \mathbb{F}_{q}^* \right\}.
\]

It is easy to see that \( \#\cH_1 = q \), \( \#\cH_2 = q-1 \), \( \mathcal{S}^{\cH_1} = \mathbb{F}_{q^4}(x) \), and
\( \mathcal{S}^{\cH_2} = \mathbb{F}_{q^4}(x^{q-1}, y^{q-1}) \). The groups \( \cH_1 \) and \( \cH_2 \) have trivial intersection, and \( \cG := \cH_1 \cH_2 \cong \cH_1 \rtimes \cH_2 \) satisfies \( \mathcal{S}^{\cG} = \mathbb{F}_{q^4}(x^{q-1}) \subseteq \mathcal{S}^{\cH_i} \) for \( i = 1, 2 \).
Let $P = (x, y)$ be a rational point on the curve such that $P$ is neither at infinity nor satisfies $x = 0$ or $y = 0$. The orbit of $P$ under the action of $\mathcal{G}$ is long. Consequently, there exist 
\[
n:=q^4 + 2q^2q_0(q-1) - 2q
\] 
rational places of $\mathcal{S}$ that are fully split over $\mathcal{S}^{\mathcal{G}}$.

Let \(0 \leq t_i, \ i = 1, 2\) and consider the $\FF_{q^4}$-vector spaces
\[
V_1 := \left\{ \sum_{k=1}^{q-1} \left( \sum_{j=0}^{t_1} a_{k,j} x^{j}  \right) y^k \mid a_{k,j} \in \mathbb{F}_{q^4} \right\},
\]
\[
V_2 := \left\{ \sum_{k=0}^{q - 3} \left( \sum_{j=0}^{t_2} a_{k,j} y^{j(q-1)} \right) x^k \mid a_{k,j} \in \mathbb{F}_{q^4} \right\}.
\]

For \(t_1  \leq q-3\), we obtain
\[
V:= V_1 \cap V_2 = \left\{  \sum_{j=0}^{t_1} a_{j} x^j y^{q-1} \mid a_{j} \in \mathbb{F}_{q^4} \right\},
\]
and \(\dim_{\mathbb{F}_{q^4}} V = t_1 + 1\). The minimum distance of the code satisfies 
\[
d \geq n - q t_1 -(q-1)(q+q_0). 
\]
For any $t_1 \leq q-3$, we  have obtained over $\FF_{q^4}$ the LRC with parameters

\[
[n, t_1+1, d\geq n- qt_1-(q-1)(q+q_0); t_1+1, 1].
\]

The example discussed here shows a number of interesting
features: first of all, the groups $\cH_1,\cH_2$ here
do not commute, but still $\cH_1\cH_2$ is a group (as follows from the observation that $\cH_1$ is normal
in the group generated by $\cH_1,\cH_2$). Note that as before, in fact we obtain an LRC with availability $1+\lfloor\frac{q-1}{t_1+1}\rfloor$. In terms of coding theory, the example
is probably not particularly useful: it can alternatively be described as the Schur-product
of the $1$-dimensional repetition code
$\textit{e}(\FF_{q^4}\cdot y^{q-1})$ and
the repetition code $\textit{e}(\sum_{j=0}^{t_1}\FF_{q^4}\cdot x^j)$.
}
\end{example}


\subsection{Subgroups with nontrivial intersection}

In the preceding subsection, we examined two subgroups whose intersection consists solely of the identity element. In this subsection, we focus on the situation where the two commuting subgroups \( \mathcal{H}_1 \) and \( \mathcal{H}_2 \) share a nontrivial intersection.

\begin{thm}\label{T4.1}
Consider the maximal curve \(\mathcal{X} / \FF_{q^2}\) defined by the equation
\[
y^{\beta} = A(x),
\]
where \(A(x)\) is an additive separable polynomial of degree \(\alpha\). 
Let \(\eta, \lambda \in \mathbb{F}^\times_{q^2}\), and define \(u := \ord(\eta)\), \(v := \ord(\lambda)\). 
Assume that \(u\) and \(v\) are divisors of \(\beta\), with \(1 < u \leq v\) and \(\gcd(u, v) = m > 1\).  
Then, for integers \(t_1 \leq (v - 2)/u\) and \(t_2 \leq \alpha - 2\), there exists an
\[
[n, (t_1+1)(t_2+1), d \geq n - \alpha u t_1 - \beta t_2; t_2+1, t_1+1]
\]
locally recoverable code over \(\FF_q\), where \(n := q^2 + (\beta - 1)(\alpha - 1)q - \alpha\).
\end{thm}

\begin{proof}
Define two subgroups of the automorphism group of \(\mathcal{X}\)  by
\[
\cH_1 := \left\{(x, y) \mapsto \left(x + a, \eta^i y\right) \mid A(a) = 0, \ i \in\mathbb{Z}\right\},
\]
and
\[
\cH_2 := \left\{ (x, y) \mapsto \left(x, \lambda^i y\right) \mid  \ i \in\mathbb{Z} \right\},
\]
so that \(\# \cH_1 = \alpha \cdot u\), \(\# \cH_2 = v\), they commute, and
\[
\cG =\cH_1 \cH_2
\]
has order \(\frac{ \alpha \cdot u \cdot v}{m}\). The fixed fields are
\[
\mathcal{F}^{\cH_1} = \mathbb{F}_{q^2}\left(A(x), y^{u}\right)=\FF_{q^2}(y^u), \quad \mathcal{F}^{\cH_2} = \mathbb{F}_{q^2}\left(x, y^{v}\right)
\]
and
\[
\mathcal{F}^{\cG} = \mathbb{F}_{q^2}\left( y^{\lcm(u,v)}\right).
\]
For \( t_i\) as in the statement of the theorem, consider the vector spaces
\[
V_1 := \left\{ \sum_{k=0}^{\alpha - 2} \left( \sum_{j=0}^{t_1} a_{k,j} y^{j u} \right) x^k \mid a_{k,j} \in \mathbb{F}_{q^2} \right\},
\]
\[
V_2 := \left\{ \sum_{k=0}^{v - 2} \left( \sum_{j=0}^{t_2} a_{k,j} x^j \right) y^k \mid a_{k,j} \in \mathbb{F}_{q^2} \right\}.
\]
Put 
\[
V:= V_1 \cap V_2 = \left\{ \sum_{k=0}^{t_1} \left( \sum_{j=0}^{t_2} a_{k,j} x^j \right) y^{k \cdot u} \mid a_{k,j} \in \mathbb{F}_{q^2} \right\},
\]
then \(\dim_{\mathbb{F}_{q^2}} V = (t_1 + 1)(t_2 + 1)\). The minimum distance of the corresponding evaluation code satisfies \(d \geq n - \alpha u t_1-\beta t_2\). Note that only $u$ of the rational
places over a totally split $Q_j$ yield the same
$y^u$ value. Therefore the assumed inequality
$t_1\leq (v-2)/u$ implies the existence of
a recovery set of size $t_1+1$ in any $\cH_2\cdot P$. This constructs a LRC over $\FF_{q^2}$ with parameters
\[
[n, (t_1+1)(t_2+1),d\geq n-\alpha u t_1-\beta t_2; t_2+1, t_1+1 ].
\]

\end{proof}

\begin{remark}\label{rm3.2.2} {\rm
   As in Remark~\ref{availability}, the proof
   presented here in fact constructs LRCs with
   availability $\lfloor\frac{\alpha u-1}{t_2+1}\rfloor + \lfloor\frac{v-1}{t_1+1}\rfloor$.
}\end{remark}

\begin{remark}\label{rem3.2.3}
    \rm{The theorem above generalizes and corrects Proposition~V.3 in \cite{BML}, where the authors considered the special case of the generalized Hermitian curve and made several mistakes in their computations, caused by (their notation, second column of \cite[p.~6806]{BML}) an
    erroneous determination of the fixed fields $\mathcal{S}^{\mathcal{H}_1}$
    and $\mathcal{S}^{\mathcal{H}_2}$. 
    Note that moreover in their statement of the proposition, the parameter $q\cdot ord(\eta)$ should be $q\cdot ord(\eta)-1$.
    
    Similar arguments as the one proving
    our result can be applied to construct codes for other maximal curves.}
\end{remark}

\begin{example} \label{ex3.2.4} \rm{
    Consider the generalized Hermitian curve \( y^{28} = x^{3} + x \) over \( \mathbb{F}_{729} \). Let \( u = 4 \), \( v = 28 \), \( t_1 = 6 \), and \( t_2 = 1 \). By Theorem~\ref{T4.1}, there exists a \( 729 \)-ary LRC with parameters \([2184, 14, d \geq 2084; 2, 7]\) and Remark~\ref{rm3.2.2} implies that this code has availability \( 8 \).  

    For comparison, see the example in \cite[Table III]{BML}. However, their construction is incorrect since \cite[Proposition~V.3]{BML} is not valid
    (see Remark~\ref{rem3.2.3} above). Their codes provide only two recovery sets, whereas our construction achieves eight recovery sets, offering significantly improved availability.}
\end{example}

 In the following, we also correct the result mentioned in \cite[Remark V.4]{BML}.

\begin{remark}
    \rm{Consider \(\eta, \lambda \in \mathbb{F}_{q^2}^\times\), where \(u:={\ord}(\eta)\) divides \(\beta\), \({\ord}(\lambda) = \beta\), and the subgroups 
\[
\mathcal{H}_1 = \left\{(x, y) \mapsto (x + a, y) \mid A(a) = 0 \right\},
\]
\[
\mathcal{H}_2 = \left\{(x, y) \mapsto (x, \lambda^i y) \mid i\in\mathbb{Z} \right\}
\]
and 
\[
\mathcal{H}_3 := \left\{(x, y) \mapsto \left(x + a, \eta^i y\right) \mid A(a) = 0, \ i \in\mathbb{Z} \right\}
\]
 of automorphisms of the curve corresponding to
 $y^{\beta}=A(x)$ (compare Theorem~\ref{T4.1}).
 
Note that \(\cG := \mathcal{H}_1\mathcal{H}_2 = \mathcal{H}_3 \mathcal{H}_2\).
With $\alpha:=\deg\,A(x)$, for integers $t_1 \leq (\beta - 2)/u$ and $t_2 \leq \alpha - 2$,  Theorem~\ref{T4.1}  provides a  LRC over \(\mathbb{F}_{q^2}\)   with
parameters
\[
[n, (t_1+1)(t_2+1), d\geq n-\alpha u t_1-\beta t_2; t_2+1, t_1+1 ].
\]

By Corollary~\ref{C4.1}, under the conditions
$t_1\leq \beta-2$ and $t_2\leq \alpha-2$ there also exists a   LRC over \(\mathbb{F}_{q^2}\)   with
parameters
\[
[n, (t_1+1)(t_2+1),d \geq n-\alpha t_1-\beta t_2; t_2+1, t_1+1].
\]
 This shows, as is also observed in \cite[Remark~V.4]{BML}, that locally recoverable codes with the same length, dimension, but 
 possibly different minimum distances may be obtained by using different vector spaces
 of functions on the same curve.
}
\end{remark}

\subsection{LRCs  from situations with three subgroups}
The next result describes situations where three
subgroups of the automorphism group of a curve are
used for constructing LRCs.

\begin{thm}\label{331}
    Let $\mathcal{F}/\mathbb{F}_{q^2}$ be a maximal function field  defined by the equations
\begin{align*}
\left\{
\begin{array}{ll}
        z^M &= B(y) \\
        y^N &= A(x),\\
\end{array}
\right.
\end{align*}
where \( A(x) \) and \( B(y) \) are additive separable polynomials of degrees \( \alpha>1 \) and \( \beta>1 \), respectively. 

Let $r\geq 2$ be a divisor of $N$, and
let $s\geq 2$ be a divisor of $M$ such that $\gcd(r,s)=1$. 
 Take  $\eta, \lambda \in \FF^\times_{q^2}$ with ${\ord}(\eta)=r$ and ${\ord}(\lambda)=s$, such that
$B(\eta y)=\eta B(y)$. 

Let \( P_\infty \) be the unique pole  of \( x \), \( y \), and \( z \in \mathcal{F} \) and write
    \[
    (x)_{\infty} = a P_{\infty} \quad \text{and} \quad (z)_{\infty} = b P_{\infty}.
    \]
Let \( \mathcal{H}_1 \), \( \mathcal{H}_2 \) and \( \mathcal{H}_3 \) be subgroups of $\textrm{Aut}_{\FF_{q^2}}(\cF)$ given by 
     \[
    \cH_1 := \left\{ (x, y, z) \mapsto (x + c, y, z) \;|\; A(c) = 0 \right\},
    \]
    \[
    \cH_2 := \left\{ (x, y, z) \mapsto (x, y, \lambda^i z) \;|\; i \in \mathbb{Z}  \right\}
    \]
    and
    \[
    \cH_3 := \left\{ (x, y, z) \mapsto (x, \eta^{iM} y, \eta^i z) \;|\; i \in \mathbb{Z} \right\}.
    \]
Then 
    \( \#\mathcal{H}_1 = \alpha \), \( \#\mathcal{H}_2 = s \) and \( \#\mathcal{H}_3 = r \).

    It holds that \( \cG := \mathcal{H}_1 \mathcal{H}_2 \mathcal{H}_3 \) is a group. Let \( n \) be the number of rational places \( P \) of \(\cF \) such that \(  \# (\cG \cdot P) =\#\cG  \). Then, for any integers \( t_1 \leq \min \{ s, r \}-2 \) and \( t_2  \leq \alpha-2 \), there exists a LRC with parameters
    \[
    [n, (t_1 + 1)(t_2 + 1), d\geq n - bt_1-at_2; t_2 + 1, t_1 + 1, t_1 + 1].
    \]

\end{thm}

\begin{proof} The assumption that $\cF$ is maximal
implies by Proposition~\ref{max1} applied to
the subfield defined by $y^N=A(x)$, that $N|(q+1)$
and all zeros of $A(x)$ are in $\FF_{q^2}$.
Hence $\eta$ as mentioned in the theorem exists.
Similarly, $\lambda$ exists and the three groups
$\cH_j$ are indeed in $\textrm{Aut}_{\FF_{q^2}}(\cF)$
and have order as stated. Since the group generated
by the $\cH_j$'s is abelian, indeed \( \cG = \mathcal{H}_1 \mathcal{H}_2 \mathcal{H}_3 \) is a group. Note that $x\not\in \mathcal{F}^{\mathcal{H}_1} = \mathbb{F}_{q^2}(y, z)$ and $z\not\in \mathcal{F}^{\mathcal{H}_2} = \mathbb{F}_{q^2}(x, y,z^s)$ and also $z\not\in \mathcal{F}^{\mathcal{H}_3}\supseteq \mathbb{F}_{q^2}(x,z^r)$.

For $t_i$ as in the statement of the theorem, consider the vector spaces
    \[
    V_1 := \left\{ \sum_{k=0}^{\alpha- 2} \left( \sum_{j=0}^{t_1} a_{k,j} z^j \right) x^k : a_{k,j} \in \mathbb{F}_{q^2} \right\},
    \]
    \[
    V_2 := \left\{ \sum_{k=0}^{s - 2} \left( \sum_{j=0}^{t_2} a_{k,j} x^j \right) z^k : a_{k,j} \in \mathbb{F}_{q^2}  \right\}
    \]
    and
    \[
    V_3 := \left\{ \sum_{k=0}^{r - 2} \left( \sum_{j=0}^{t_2} a_{k,j} x^j \right) z^k : a_{k,j} \in \mathbb{F}_{q^2}  \right\}.
    \]

    Then, we obtain that
    \[
    V := V_1 \cap V_2 \cap V_3 = \left\{ \sum_{k = 0}^{t_1} \sum_{j= 0}^{t_2} a_{k, j} z^{k} x^{j} : a_{k, j} \in \mathbb{F}_{q^2}  \right\}
    \]
    is contained in \( L((bt_1+at_2) P_{\infty}) \) and has dimension
    \[
    \dim_{\mathbb{F}_{q^2} } V = (t_1 + 1)(t_2 + 1).
    \]
Evaluating at the $n$ rational places \( P \) of \(\cF \) such that \(  \# \cG \cdot P =\#\cG  \) yields a LRC over \(\mathbb{F}_{q^2}\) with parameters
    \[
    [n, (t_1 + 1)(t_2 + 1), d\geq n - bt_1-at_2; t_2 + 1, t_1 + 1, t_1 + 1].
    \]
\end{proof}

\begin{remark}
{\rm
   As in Remark~\ref{availability}, the proof
   presented here in fact constructs LRCs with
   availability $\lfloor\frac{\alpha-1}{t_2+1}\rfloor + \lfloor\frac{r-1}{t_1+1}\rfloor+ \lfloor\frac{s-1}{t_1+1}\rfloor$.
}\end{remark}

\begin{example}
    \rm{
Let $\mathcal{X}/\mathbb{F}_{q^6}$ be the Giulietti-Korchmáros curve  defined by the following equation:
\begin{align*}
\left\{
\begin{array}{ll}
        z^{q^2-q+1} &= y^{q^2}-y \\
        y^{q+1} &= x^q+x.\\
\end{array}  
\right.
\end{align*}
It is well known that $\mathcal{X}$ is a maximal curve over $\mathbb{F}_{q^6}$, having precisely $1 + q^8 - q^6 + q^5$ rational points. 
 There exists a unique point $P_{\infty}$ where $x, y$ and $z$ have a pole. One obtains
    \[
    (x)_{\infty} = (q^3+1) P_{\infty} \quad \text{and} \quad (z)_{\infty} = q P_{\infty}.
    \]

Consider positive integers $r|(q+1)$ and $s|(q^2 - q + 1)$. For any $t_1 \leq \min\{r, s\}-2$ and $t_2 \leq q-2$, Theorem~\ref{331} gives a LRC over $\FF_{q^6}$ with parameters
\[
    [n, (t_1 + 1)(t_2 + 1), d\geq n - qt_1-(q^3+1)t_2; t_2 + 1, t_1 + 1, t_1 + 1]
    \]
where $n=q^8-q^6+q^5-q^3.$ The choice \( r = q+1 \), \( s = q^2 - q + 1 \), \( t_1 = q-1 \), and \( t_2 = q-2 \) provides a LRC with availability \( q+1 \).

}
\end{example}

\begin{remark}\rm{
   The theorem and example above generalize Proposition IV.6 in \cite{BML}, which focuses on the Giulietti-Korchmáros curve. While their construction yields a code with availability 3, our example provides a code with availability $q + 1$.
}
\end{remark}

\noindent{\textbf{Acknowledgment.} \rm{ The first author was partially supported by FAPESP grant No.~2023/08271-5.}}

\end{document}